\documentclass{amsart}
\usepackage{graphicx}
\usepackage{amssymb}
\usepackage{enumerate}
\usepackage{enumitem}
\usepackage{mathtools}
\newtheorem{main}{Theorem}

\newtheorem{theorem}{Theorem}[section]
\newtheorem{corollary}{Corollary}[theorem]
\newtheorem{lemma}[theorem]{Lemma}

\theoremstyle{definition}

\DeclareMathSymbol{\shortminus}{\mathbin}{AMSa}{"39}
\newcommand{\medminus}{\shortminus\!}

    \title{Cassini sets in taxicab geometry}
    
    \author{Alexander Habib}
    \address{Department of Mathematics, Ohio State University, USA}
    \email{habib.84@osu.edu}
    
    \author{Dylan Helliwell}
    \address{Department of Mathematics, Seattle University, USA}
    \email{helliwed@seattleu.edu}
    \thanks{We are grateful to our reviewer who provided helpful suggestions for improvement of this paper.  Any remaining errors are our own.}

    \date{\today}
    
\keywords{Taxicab geometry, Cassini Sets}
\subjclass[2020]{51M05, 51M15}

\begin{document}
    
    \begin{abstract}
        Given two points $p$ and $q$ in the plane and a nonnegative number $r$, the Cassini oval is the set of points $x$ that satisfy $d(x, p) d(x, q) = r^2$.  In this paper, we study this set using the taxicab metric.  We find that these sets have characteristics that are qualitatively similar to their Euclidean counterparts while also reflecting the underlying taxicab structure.  We provide a geometric description of these sets and provide a characterization in terms of intersections and unions of a restricted family of such sets analogous to that found recently for taxicab Apollonian sets.
    \end{abstract}
    \maketitle
    
    \section{Introduction}        
        The Taxicab plane is $(\mathbb{R}^2, d)$ where $d$ is the taxicab distance function
        \begin{equation*}
            d(p, q) = |p_1 - q_1| + |p_2 - q_2|.
        \end{equation*}
        
This space is well studied as an alternative to the Euclidean plane, and a number of interesting geometric properties have been explored.  See for example \cite{Krause, Reynolds, Laatsch, BCFHMNSTV, FHS} and the references therein for a variety of examples.

Conic sections have been a particular area of interest \cite{Reynolds, KAGO} and recently, Apollonian sets have been studied \cite{BCFHMNSTV}.  Together, these sets arise by requiring that the sum, difference, or ratio of distances to two fixed points be constant.
        
       In this paper, we add to this family by analyzing sets that are defined by requiring that the product of taxicab distances to two fixed points be constant.  More formally, given two points $p$ and $q$ in $\mathbb{R}^2$, and  $r \in [0,\infty)$, we explore the Cassini set
        \begin{equation*}
            K(p, q; r) = \left\{ x \in \mathbb{R}^2: d(x, p) \cdot d(x, q) = r^2 \right\}.
        \end{equation*}
Such sets have been explored somewhat in \cite{MartiniWu, JahnMartiniRichter} where analytical properties of Cassini sets and others were studied using a variety of norms and gauges.

        Using the Euclidean distance $d_E$, such curves, shown in Figure~\ref{euclideancassinifig}, arise in a variety of settings, making their first appearance in classical Greek geometry through the study of spiric sections, a specific family of toric sections, and then later as analytic geometry developed \cite{Stillwell}.  These curves also arise in physical settings.  A contemporary of Sir Isaac Newton, Jean-Dominique Cassini \cite{Cassini} studied convex representatives as a possible model for celestial motion, hence their name; and more recently, Cassini ovals have appeared in examples from electrostatics \cite{Jackson}.

        While the algebraic equations describing ellipses, hyperbolas, and Apollonian circles are quadratic, those for Euclidean Cassini ovals are quartic, doubling the degree of the polynomials required.  Letting $r_E^* = \frac{1}{2} d_E(p, q)$, if $r < r_E^*$, the Cassini oval comprises two simple closed curves.  When $r = r_E^*$ the set is a lemniscate, and for $r > r_E^*$ it is a single simple closed curve.  If $r \geq \sqrt{2} r_E^*$, the curve is convex.  If $d_E(p, q) = 0$, the set is a circle.

        Shifting focus to the taxicab setting, we find that while Cassini sets reflect the taxicab metric in their specific shape, see Figure~\ref{cassinifig}, they nonetheless share some qualitative similarities with their Euclidean counterparts.  Specifically, while taxicab ellipses, hyperbolas, and Apollonian sets are piecewise linear, taxicab Cassini sets have portions that are quadratic, again doubling the degree.  Also, taxicab Cassini sets experience a transition at $r = r^* = \frac{1}{2} d(p, q)$ from a pair of simple closed curves to a single simple closed curve.  Unlike the Euclidean case, taxicab Cassini sets are never convex in the geodesic sense and convex in the linear sense only when $d(p, q) = 0$ and the set is a taxicab circle.  These observations are formalized in our first main result, Theorem~\ref{descriptionthm} in Section~\ref{descriptionsec}.
        
        \begin{figure}
            \begin{picture}(200,140)
                \put(0,-10){
                    \includegraphics[scale = .8, clip = true, draft = false]{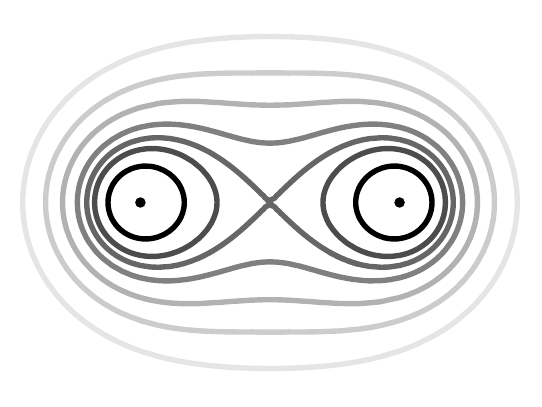}
                }
                \put(60,67){$p$}
                \put(148,67){$q$}
            \end{picture}
            
            \caption{
                Cassini ovals in Euclidean space with lighter shading corresponding to increasing $r$.  The second-from-outermost curve corresponds to $r = \sqrt{2}\,r_E^*$.
            }\label{euclideancassinifig}
        \end{figure}
          
        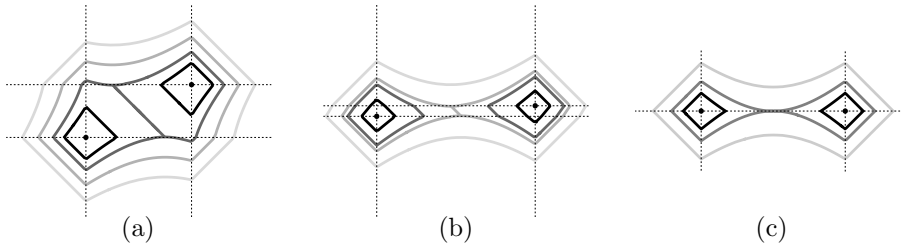
\begin{figure}
        \begin{picture}(360,80)
        
        \put(5,-10){
        \includegraphics[scale = .4, clip = true, draft = false]{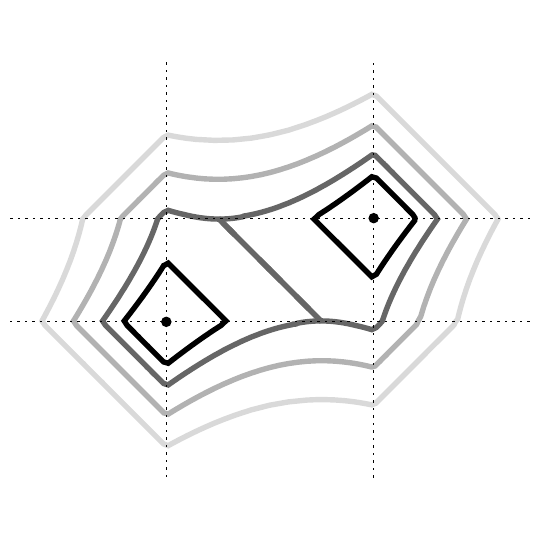}
        \put(-59,5){(a)}
        }
        
        \put(125,-10){
        \includegraphics[scale = .4, clip = true, draft = false]{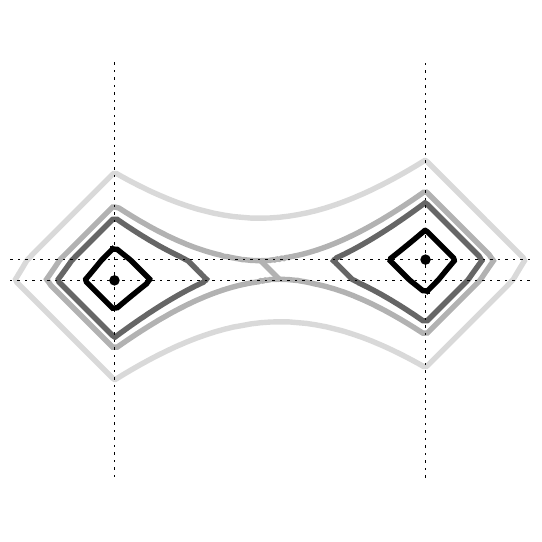}
         \put(-59,5){(b)}
        }
        
        \put(245,-10){
        \includegraphics[scale = .4, clip = true, draft = false]{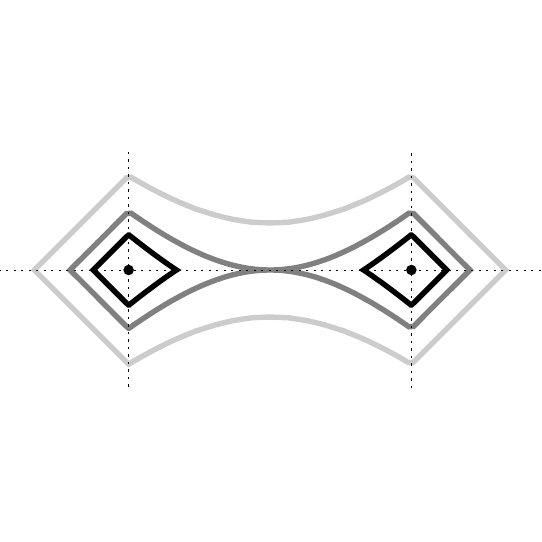}
        \put(-59,5){(c)}
        }        
        \end{picture}
        
        \caption{Cassini sets when the foci are in general position (a) and (b), and when the foci share a coordinate line (c). As in the Euclidean case, a transition from a pair of simple closed curves to a single simple closed curve occurs at $r = r^* = \frac{1}{2} d(p, q)$.} \label{cassinifig}
        \end{figure}               
        
        In somewhat of a surprise, in \cite{BCFHMNSTV}, it was shown that taxicab Apollonian sets could be characterized in terms of unions of those whose focal points lie on lines with slope $\pm 1$.  We find that taxicab Cassini sets enjoy a similar characterization.  Lines of slope $\pm 1$ play an important role in taxicab geometry so we call them guide lines, and we say a \textit{guide Cassini set} is a Cassini set whose focal points share a guide line.  With this, we are able to show that any Cassini set can be characterised in terms of four specific guide Cassini sets, and that the characterization can occur in two different ways.
        Together, these are presented as Theorems~\ref{unionofintersectionsthm} and \ref{intersectionofunionsthm} in Section~\ref{characterizationsec}.
        
        We conclude the paper with some remarks about emerging themes and tthe ways in which this work may be extended in Section~\ref{conclusionsec}.

    
    \section{Supporting structure} \label{backgroundsec}

        We provide here a variety of structural definitions and establish some background results to support our main results.

        \subsection{Points and lines}
            Given a point $p = (p_1, p_2)$ we define the \textit{coordinate lines} through $p$ to be the lines
            \begin{align*}
                cl^1(p) &= \{x \in \mathbb{R}^2:
                    x_1 = p_1\}, \\
                cl^2(p) &= \{x \in \mathbb{R}^2:
                    x_2 = p_2\},
            \end{align*}
            and we define the \textit{guide lines} through $p$ to be the lines
            \begin{align*}
            gl^+(p) &= \{x \in \mathbb{R}^2:
                x_2 - p_2 = x_1 - p_1 \}, \\
             gl^-(p) &= \{x \in \mathbb{R}^2:
                x_2 - p_2 = -(x_1 - p_1) \}.
            \end{align*}
            Given two points $p$ and $q$, we define the \textit{coordinate complements} to be
            \begin{align*}
                c^1(p, q) &= cl^1(p) \cap cl^2(q)
                    = (p_1, q_2), \\
                c^2(p, q) &= cl^2(p) \cap cl^1(q)
                    = (q_1, p_2),               
            \end{align*}
            and we define the \textit{guide complements} to be
            \begin{align*}
                g^+(p, q) &= gl^+(p) \cap gl^-(q), \\
                g^-(p, q) &= gl^-(p) \cap gl^+(q).
            \end{align*}
            See Figure~\ref{structurefig}(a) for representations of these lines and points, and \cite{BCFHMNSTV} for more detail about them.  Finally, we define the \textit{midpoint} of $p$ and $q$ as follows:
            \[
            m(p, q) = \left(\frac{p_1 + q_1}{2}, \frac{p_2 + q_2}{2} \right).
            \]
     
            \begin{figure}
                \begin{picture}(380,100)
                    
                    \put(-5,10){
                        \includegraphics[scale = .4, draft = false]{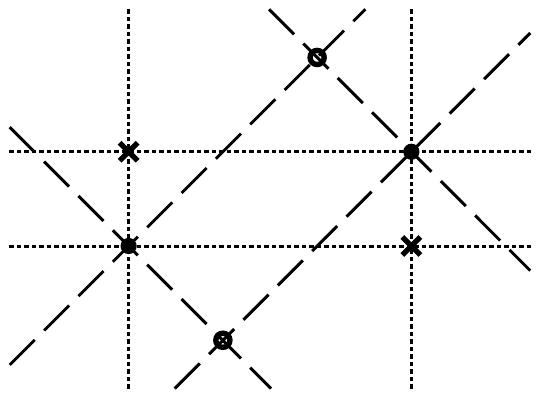}
                        \put(-24,55){$p$}
                        \put(-85,19){$q$}
                        \put(-24,19){$c^1$}
                        \put(-77,49){$c^2$}
                        \put(-64,17){$g^+$}
                        \put(-44,72){$g^-$}
                        \put(-60,-13){(a)}
                    }
                                        
                    \put(125,20){
                        \includegraphics[scale = .3, draft = false]{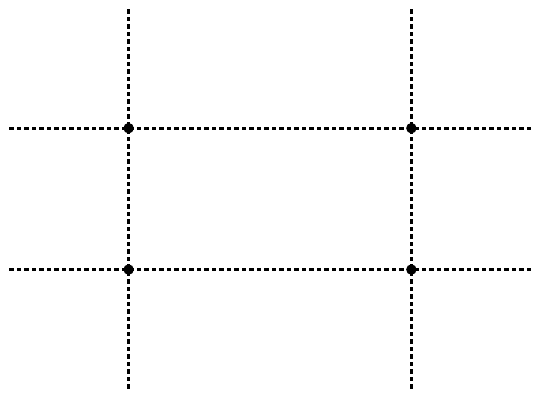}
                        \put(-58,13){\footnotesize{$q$}}
                        \put(-25,42){\footnotesize{$p$}}
                        \put(-58,41){\footnotesize{$c^2$}}
                        \put(-27,9){\footnotesize{$c^1$}}
                        \put(-77, 48){$Q_{c^2}$}
                        \put(-49, 48){$S_{p, c^2}$}
                        \put(-13, 48){$Q_p$}
                        \put(-80, 26){$S_{q, c^2}$}
                        \put(-44, 26){$R$}
                        \put(-16, 26){$S_{p, c^1}$}
                        \put(-77, 4){$Q_q$}
                        \put(-49, 4){$S_{q, c^1}$}
                        \put(-13, 4){$Q_{c^1}$}
                        \put(-45,-21){(b)}
                    }                    
                    
                    \put(253,10){
                        \includegraphics[scale = .4, draft = false]{regions.pdf}
                        \put(-78,18){\footnotesize{$q$}}
                        \put(-32,56){\footnotesize{$p$}}
                        \put(-128,67){\footnotesize{$\sigma_p\!=\!(\medminus1,1)$}}
                        \put(-128,57){\footnotesize{$\sigma_q\!=\!(\medminus1,1)$}}
                         \put(-76,67){\footnotesize{$\sigma_p\!=\!(\medminus1,1)$}}
                        \put(-76,57){\footnotesize{$\sigma_q\!=\!(1,1)$}}
                         \put(-23,67){\footnotesize{$\sigma_p\!=\!(1,1)$}}
                        \put(-23,57){\footnotesize{$\sigma_q\!=\!(1,1)$}}
                        \put(-128,41){\footnotesize{$\sigma_p\!=\!(\medminus1,\medminus1)$}}
                        \put(-128,31){\footnotesize{$\sigma_q\!=\!(\medminus1,1)$}}
                        \put(-74,41){\footnotesize{$\sigma_p\!=\!(\medminus1,\medminus1)$}}
                        \put(-74,31){\footnotesize{$\sigma_q\!=\!(1,1)$}}
                        \put(-23,41){\footnotesize{$\sigma_p\!=\!(1,\medminus1)$}}
                        \put(-23,31){\footnotesize{$\sigma_q\!=\!(1,1)$}}
                        \put(-128,14){\footnotesize{$\sigma_p\!=\!(\medminus1,\medminus1)$}}
                        \put(-128,4){\footnotesize{$\sigma_q\!=\!(\medminus1,\medminus1)$}}
                        \put(-72,14){\footnotesize{$\sigma_p\!=\!(\medminus1,\medminus1)$}}
                        \put(-72,4){\footnotesize{$\sigma_q\!=\!(1,\medminus 1)$}}
                         \put(-23,14){\footnotesize{$\sigma_p\!=\!(1,\medminus1)$}}
                        \put(-23,4){\footnotesize{$\sigma_q\!=\!(1,\medminus1)$}}
                        \put(-60,-13){(c)}
                    }
                \end{picture}
                    \caption{
                        (a) Coordinate lines, dotted; guide lines, dashed; coordinate complements and guide complements.  (b) The regions defined by $p$ and $q$.  (c) The values of $\sigma_p$ and $\sigma_q$ in each region.} \label{structurefig}
            \end{figure}
            
        \subsection{Regions}            
            Together, the two points, $p$ and $q$, their coordinate complements $c^1$ and $c^2$, and the coordinate lines through them divide the plane into nine regions.  These regions comprise four \textit{quadrants} $Q_p$, $Q_q$, $Q_{c^1}$, $Q_{c^2}$, four \textit{half-strips} $S_{p, c^1}$, $S_{p, c^2}$, $S_{q, c^1}$, $S_{q, c^2}$, and a \textit{central rectangle} $R$, also called the \textit{coordinate rectangle}.  These regions are defined to be including their boundaries so that two half-strips and the central rectangle collapse to rays and a segment if $p$ and $q$ share a coordinate line.  See Figure~\ref{structurefig}(b).

            Comparing to the notation used in \cite{BCFHMNSTV}, the quadrants are the regions $R_1$, $R_3$, $R_7$, $R_9$, the half-strips are the regions $R_2$, $R_4$, $R_6$, $R_8$, and the central rectangle is $R_5$.

            Let
            \begin{equation*}
                \sigma_{pj}(x) =
                    \begin{cases}
                        \frac{x_j-p_j}{|x_j - p_j|}
                            & \mbox{if} \ x_j \neq p_j \\
                        0 & \mbox{if} \ \ x_j = p_j
                    \end{cases}
            \end{equation*}
            and let $\sigma_p(x) = \Bigl(\sigma_{p1}(x), \sigma_{p2}(x)\Bigr)$.
            We can use these functions to characterize the various regions; see Figure~\ref{structurefig}(c).

            Note that, given two points $p$ and $q$, the central rectangle $R$ can also be characterized as the set of points $x$ with the property that $d(x, p) + d(x, q) = d(p, q)$.  In other words, $R$ is the set of points in the taxicab plane where the triangle inequality is actually an equality.

        \subsection{Relative proximity}
            Given points $p$ and $q$, Let
            \[
            E(p, q) = \{x \in \mathbb{R}^2:d(x, p) = d(x, q)\}
            \]
            and let
            \[
            H(p, q) = \{x \in \mathbb{R}^2: d(x, p) \leq d(x, q)\}.
            \]
            See Figure~\ref{e-and-h-setsfig}.  The sets $E(p, q)$ were explored in detail in \cite{Reynolds} and \cite{KAGO}.  More recently in \cite{BCFHMNSTV}, both $E(p, q)$ and $H(p, q)$ were described in terms of Apollonian sets.

            The following lemma, which is essentially a restatement of Lemma~3.14 from \cite{BCFHMNSTV},  will prove to be quite useful:
            \begin{lemma} \label{guideHlemma}
                Let $b$ and $c$ lie on distinct guide lines through $a$.  Then
                \[
                H(a, b) \cup H(a, c) = \mathbb{R}^2.
                \]
            \end{lemma}
            
            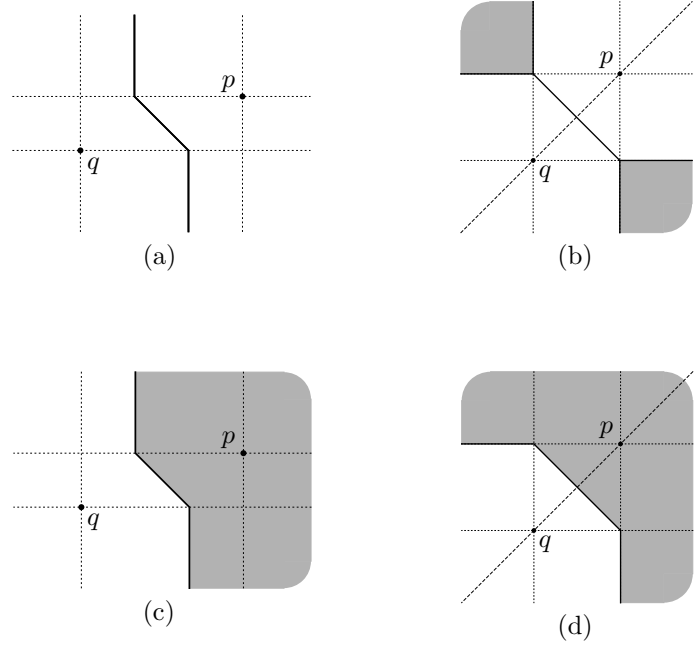
\begin{figure}
                \begin{picture}(280,240)
                    \put(0,150){
                        \includegraphics[scale = .45, draft = false]{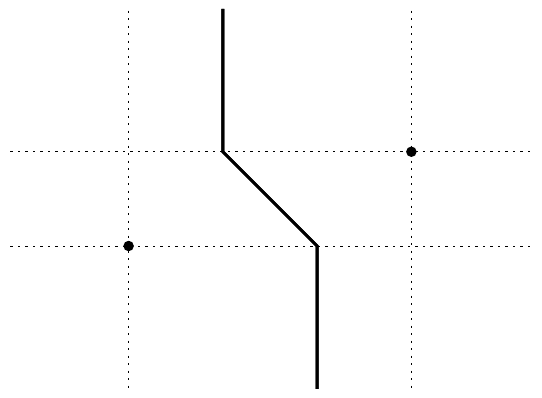}
                        \put(-87,26){$q$}
                        \put(-36,57){$p$}
                        \put(-66,-10){(a)}
                    }
                    \put(170,150){
                        \includegraphics[scale = .35, draft = false]{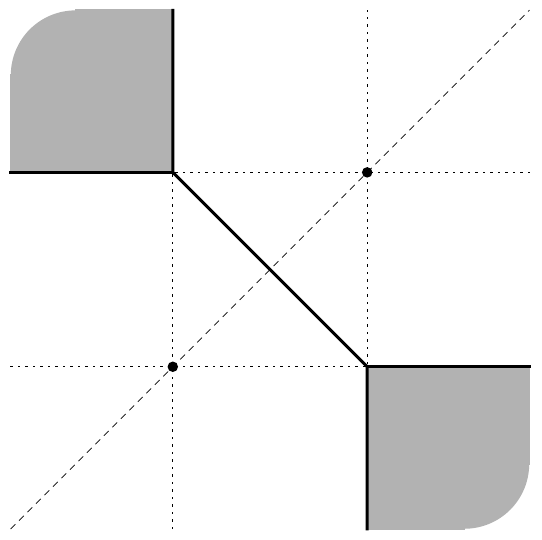}
                        \put(-60,23){$q$}
                        \put(-37,66){$p$}
                        \put(-53,-10){(b)}
                    }
                    \put(0,15){
                        \includegraphics[scale = .45, draft = false]{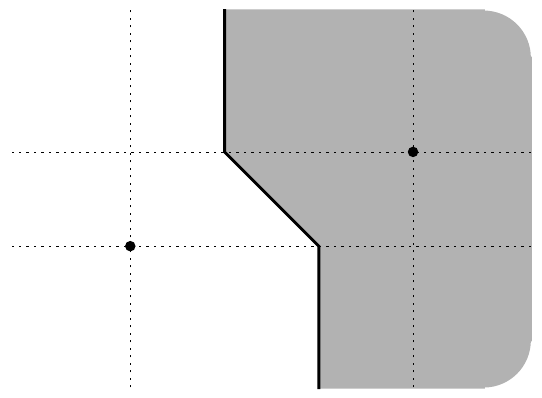}
                        \put(-87,26){$q$}
                        \put(-36,57){$p$}
                        \put(-66,-10){(c)}
                    }
                    \put(170,10){
                        \includegraphics[scale = .35, draft = false]{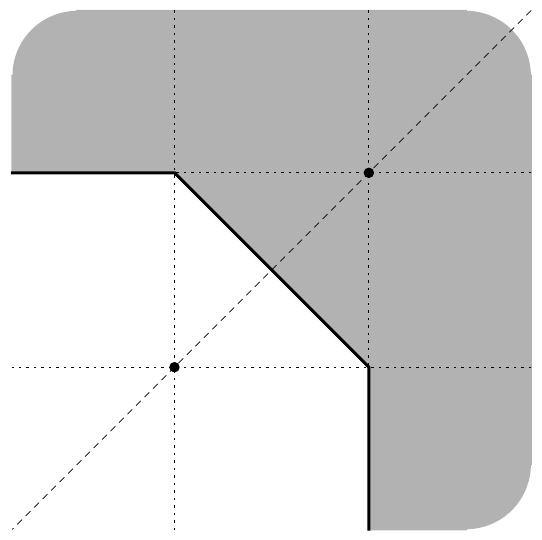}
                        \put(-60,23){$q$}
                        \put(-37,66){$p$}
                        \put(-53,-10){(d)}
                    }
                \end{picture}
                
                \caption{
                    The sets $E(p, q)$ when $p$ and $q$ do not share a guide line (a) and when $p$ and $q$ do share a guide line (b), and the sets $H(p, q)$ when $p$ and $q$ do not share a guide line (c) and when $p$ and $q$ do share a guide line (d).
                } \label{e-and-h-setsfig}
            \end{figure}

        \subsection{Isometries}
            As shown in \cite{Schattschneider}, The isometry group for the taxicab plane is isomorphic to $\mathbb{R}^2 \rtimes D_4$, with the $\mathbb{R}^2$ factor corresponding to translations and the $D_4$ factor corresponding to rotations about a point by multiples of $\frac{\pi}{2}$ and reflections across guide lines and coordinate lines through the point.

            We say two points $a$ and $b$ are in \textit{standard position} if $a$ is the origin and $b$ is in the first quadrant with $b_2 \leq b_1$.  We note that given any two points $a$ and $b$, there is a taxicab isometry $\varphi$ such that $\varphi(a)$ and $\varphi(b)$ are in standard position.
            
    
    \section{Geometric description of Cassini sets} \label{descriptionsec}
        In this section, we establish a geometric description of taxicab Cassini sets.  Cassini sets are isometry invariants in the sense that if $\varphi$ is an isometry, then
        \[
            \varphi\Bigl(K(p, q; r)\Bigr) = K\Bigl(\varphi(p), \varphi(q); r\Bigr).
        \]
        Additionally, they are scale invariants as long as the parameter $r$ is adjusted appropriately.  If $\delta$ is a dilation by a factor of $\lambda$ then
        \[
            \delta\Bigl(K(p, q; r)\Bigr) = K\Bigl(\delta(p), \delta(q); \lambda r\Bigr).
        \]
        With this in mind, we provide explicit algebraic formulas in the specific case where $m(p,q)$ and $p$ are in standard position.  This ensures that $q = -p$ which simplifies and clarifies the structure.
        
        Taxicab isometries preserve the geometric properties arising from the algebraic formulations, thus establishing a qualitative description which applies for $p$ and $q$ in general position.
        
        Portions of taxicab Cassini sets consist of hyperbolic segments.  To simplify the description, let $\mathcal{H}(p; r)$ be the pair of Euclidean hyperbolas centered at $p$, having asymptotes equal to the guide lines through $p$, and with vertices at the points $p \pm (r, 0)$ and $p \pm (0, r)$.

        \begin{main} \label{descriptionthm}
            Let $K = K(p, q; r)$ be the taxicab Cassini set defined by $p, q \in (\mathbb{R}^2, d)$, and $r \geq 0$.  Then
            \begin{itemize}
                \item for each quadrant $Q$ defined by $p$ and $q$, $K \cap Q$ is either empty or a segment on a guide line ending at points on the coordinate lines defining $Q$;
                \item in the central rectangle $R$, $K \cap R$ is either empty, a single segment, or a pair of parallel segments, which always lie on guide lines parallel to $E(p, q) \cap R$;
                \item for each half-strip $S$ defined by $p$ and $q$, $K \cap S = \mathcal{H}(g, r) \cap S$ where $g$ is the guide complement of $p$ and $q$ with the property that both of the guide lines defining it avoid the interior of $S$.
            \end{itemize}
            These sets coincide on the shared coordinate lines between two regions and together, they form two simple closed curves if $r < r^* = \frac{1}{2} d(p, q)$ and a single simple closed curve if $r > r^*$.  If $r = r^*$, $K$ is homeomorphic to two squares sharing an edge if $p$ and $q$ are not on the same coordinate line, or two squares sharing only a vertex if $p$ and $q$ are on a coordinate line.
            
            If $m(p, q)$ and $p$ are in standard position, then $K$ is defined by the following equations:
            \begin{itemize}
            \item In the interior of each quadrant $Q$
            \begin{equation} \label{Qeqn}
            \sigma_{p1}\, x_1 + \sigma_{p2}\, x_2
                = \sqrt{(\sigma_{p1}\, p_1
                            + \sigma_{p2}\, p_2)^2 + r^2}.
            \end{equation}
            \item In the interior of the central rectangle $R$
            \begin{equation} \label{Reqn}
            x_1 + x_2
                = \pm \sqrt{(p_1 + p_2)^2
                        -r^2}.
            \end{equation}
            \item In the interior of each half-strip $S$
            \begin{equation} \label{Seqn}
            (x_1 - \sigma_{p1} \sigma_{p2}\, p_2)^2
                - (x_2 - \sigma_{p1} \sigma_{p2}\, p_1)^2
                    = \sigma_{p1} \sigma_{q1}\, r^2.
            \end{equation}
            \end{itemize}
        \end{main}

        See Figure~\ref{hyperbolafig} for examples of the behavior in each region.
        
        \begin{figure}
            \begin{picture}(280,210)
                \put(0,110){
                    \includegraphics[scale = .5, clip = true, draft = false]{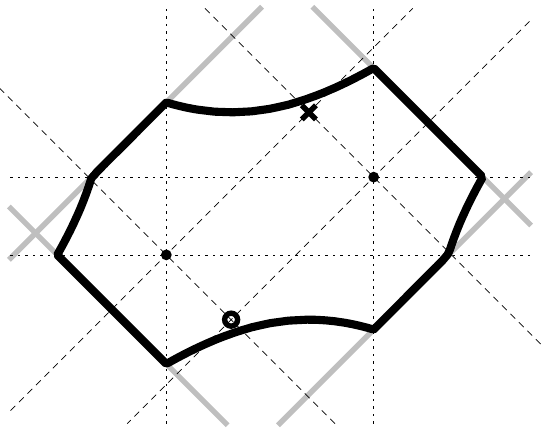}
                    \put(-100,42){\small{$q$}}
                    \put(-36,61){\small{$p$}}
                    \put(-59,69){\small{$g^-$}}
                    \put(-80,32){\small{$g^+$}}
                }
                \put(160,115){
                    \includegraphics[scale = .45, clip = true, draft = false]{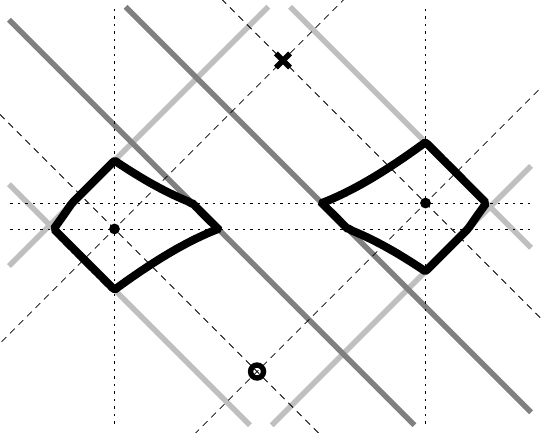}
                    \put(-89,44){\small{$q$}}
                    \put(-34,48){\small{$p$}}
                    \put(-60,72){\small{$g^-$}}
                    \put(-64,20){\small{$g^+$}}
                }
                \put(0,0){
                    \includegraphics[scale = .5, clip = true, draft = false]{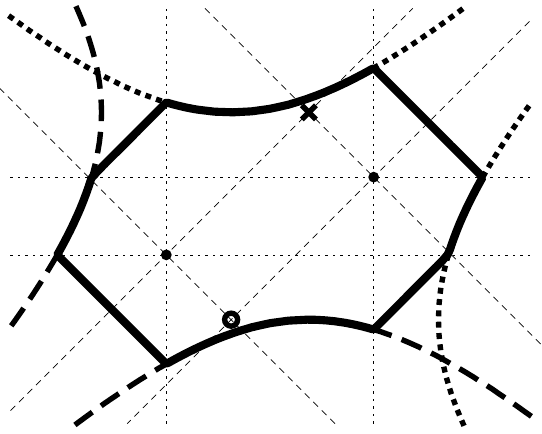}
                    \put(-100,42){\small{$q$}}
                    \put(-36,61){\small{$p$}}
                    \put(-59,69){\small{$g^-$}}
                    \put(-80,32){\small{$g^+$}}
                }    
                \put(160,5){
                    \includegraphics[scale = .45, clip = true, draft = false]{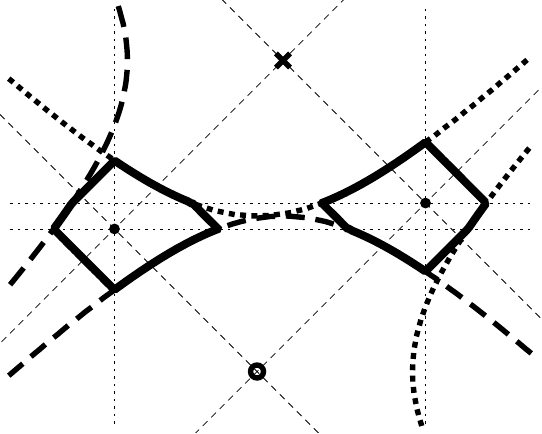}
                    \put(-89,44){\small{$q$}}
                    \put(-34,48){\small{$p$}}
                    \put(-60,72){\small{$g^-$}}
                    \put(-64,20){\small{$g^+$}}
                }
            \end{picture}
            
            \caption{
                The Cassini set $K(p, q; r)$ intersects the quadrants and central rectangle along guide lines (top row) and intersects half-strips in hyperbolas (bottom row).  Dashed curves are centered at $g^-$ and dotted curves are centered at $g^+$.  On the left, $p = (8, 3)$, $q = -p$, and $r = 16$.  On the right, $p = (12, 1)$, $q = -p$, and $r = 12$.
            } \label{hyperbolafig}
        \end{figure}
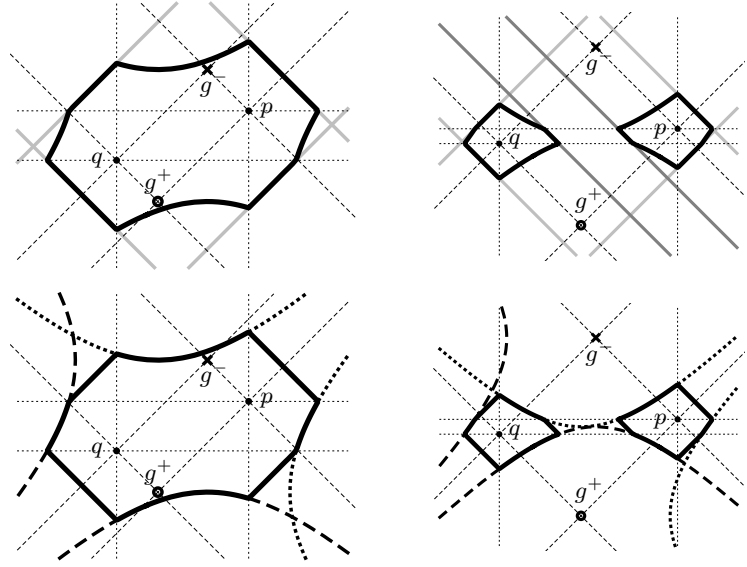

        \begin{proof}
            The geometric descriptions of $K$ follow from the algebraic formulas, which are established by resolving the absolute values in the definition of the taxicab Cassini set $K(p,q; r)$ according to the region in which $x$ lies.  With $m(p, q)$ and $p$ in standard position so that $q = -p$, the definition simplifies to
            \begin{equation} \label{newstructureeqn}
            \begin{split}
                \sigma_{p1} \sigma_{q1} &(x_1^2 - p_1^2)
                    + \sigma_{p2} \sigma_{q2}(x_2^2 - p_2^2) \\
                    & + \sigma_{p1} \sigma_{q2}
                    (x_1 x_2 - p_1 x_2 + p_2 x_1 - p_1 p_2) \\
                    & + \sigma_{p2} \sigma_{q1}
                        (x_1 x_2 + p_1 x_2 - p_2 x_1 - p_1 p_2)
                            = r^2.
            \end{split}
            \end{equation}
            In the interior of any quadrant $Q$, $\sigma_q = \sigma_p$, so Equation~\eqref{newstructureeqn} reduces to
            \[
                (\sigma_{p1}\, x_1 + \sigma_{p2}\, x_2)^2
                    = (\sigma_{p1}\, p_1
                        + \sigma_{p2}\, p_2)^2 + r^2.
            \]
            Taking the square root and noting that the resulting left hand side is always positive results in Equation~\eqref{Qeqn}, which is the equation for a guide line that intersects the coordinate lines defining the quadrant.  Note that this line is always defined, and always contributes to $K$ in $Q_p$ and $Q_q$, but does not always intersect $Q_{c^1}$ and $Q_{c^2}$.  In these quadrants, it only contributes to $K$ if $r \geq \sqrt{|p_1-q_1|\cdot |p_2 - q_2|}$.

            In the interior of the rectangle $R$, $\sigma_p = (-1, -1)$ and $\sigma_q = -\sigma_p$, so Equation~\eqref{newstructureeqn} reduces to
            \[
                (x_1 + x_2)^2 = (p_1 + p_2)^2 - r^2.
            \]
            Taking the square root results in Equation~\eqref{Reqn}.  In this case, both the positive and negative square root are possible.  Note that $p_1 + p_2 = \frac{1}{2} d(p, q)$, so if $r < r^*$, there are two lines.  When $r = r^*$, there is one solution passing through $m(p, q)$.  When $r > r^*$ there are no solutions.

            In the interiors of the half-strips $S$, among $\sigma_{p 1}$, $\sigma_{p 2}$, $\sigma_{q 1}$, and $\sigma_{q 2}$, exactly three are equal to each other.  This implies that the product of any two will always be the negative of the product of the other two.  This in turn implies that Equation~\eqref{newstructureeqn} reduces to Equation~\eqref{Seqn}, which is that of hyperbolas centered at $g^+$ or $g^-$ and having guide lines as asymptotes.  In the two half-strips $S_{p, c^1}$ and $S_{p, c^2}$, $\sigma_{p1} \sigma_{p2} = -1$, which implies that the center of the resulting hyperbolas is at the point $(-p_2, -p_1) = g^+$ both of whose guide lines avoid the interior of $S_{p, c^1}$ and $S_{p, c^2}$.  Similarly, for the two half-strips touching $q$, the center is $g^-$.  For a given guide complement $g$, the product $\sigma_{p1} \sigma_{q1}$ specifies which of the two hyperbolas of $\mathcal{H}(g, r)$ is represented by Equation~\eqref{Seqn}, and verifies that for a given half-strip $S$, $K \cap S = \mathcal{H} \cap S$.
        \end{proof}    
            
    
    \section{Characterization of Cassini sets} \label{characterizationsec}
        In this section we characterize a Cassini set $K(p, q; r)$ in terms of a family of guide Cassini sets determined by $p$, $q$, and $r$.  To do this, we introduce the filled Cassini set
        \begin{equation*}
            L(p,q;r) = \left\{x\in \mathbb{R}^2 : d(x,p) \cdot d(x,q) < r^{2}\right\}.
        \end{equation*}
        We start with a technical lemma establishing that a Cassini set is the boundary of the corresponding filled Cassini set, and the strict inequality in the definition of filled Cassini sets is chosen to simplify this analysis.

        \begin{lemma} \label{Kispdyllemma}
            Let $p, q \in (\mathbb{R}^2, d)$, and let $r > 0$.  Then
            $K(p,q;r) = \partial L(p,q;r)$.
        \end{lemma}

        Note that if $r = 0$ then $K(p, q) = \{p, q\}$ and $L(p, q)$ is empty.

        \begin{proof}
            For a given pair of points $p$ and $q$, define $f:\mathbb{R}^2 \rightarrow \mathbb{R}$, $f(x) = d(x, p) \cdot d(x, q)$.  Note that the level sets of $f$ are the sets $K(p, q; r)$.  Moreover, $f$ is continuous, so $L(p, q; r) = \{x \in \mathbb{R}^2:f(x) < r^2\}$ and $\{x \in \mathbb{R}^2:f(x) > r^2\}$ are both open.  Hence, if $x \in \partial L(p, q; r)$, then $f(x)$ cannot be less than $r^2$ nor greater than $r^2$.  Therefore $f(x) = r^2$, so $x \in K(p, q; r)$ and so $\partial L(p, q; r) \subseteq K(p, q; r)$.

            The other direction is a bit more subtle.  For a given $x \in K(p, q; r)$, if there is a partial derivative of $f$ that is non-zero at that point, then $x$ must also lie in $\partial L(p, q; r)$.  This condition is met everywhere except at $p$, $q$, and $E(p, q) \cap R$.

            Specifically, if $x$ does not lie in a coordinate line through $p$ or $q$ then $f$ is smooth and the gradiant is nonzero as long as $x \notin E(p, q) \cap R$.
            On a coordinate line through $p$ or $q$, $f$ is not smooth and the gradient cannot be computed, but the partial derivative in the direction of the coordinate line is still well defined and nonzero for all $x$ other than $p$ and $q$, which correspond to $r = 0$ and are outside the scope of the lemma, and the two points in $E(p, q) \cap \partial R$.
                        
            For points in $E(p, q) \cap R$ the gradient is zero and the analysis above does not apply.  By Theorem~\ref{descriptionthm}, this set is a line segment corresponding to the critical parameter $r^* = \frac{1}{2} d(p, q)$ and $K(p, q; r^*) \cap R = E(p, q) \cap R$.
            For $x \in R$, the triangle inequality becomes an equality, so the AMGM inequality yields
            $r^* \geq \sqrt{d(x, p) \cdot d(x, q)}$.
            Hence if $x \in R$, $f(x) \leq (r^*)^2$, and  therefore if $x$ is an element of $E(p, q) \cap R$, any neighborhood of that point includes points in $L(p, q; r^*)$.
        \end{proof}

        Because of Lemma~\ref{Kispdyllemma}, characterizing filled Cassini sets also characterizes the corresponding Cassini sets.  With this in mind, next we prove Theorems~\ref{unionofintersectionsthm} and \ref{intersectionofunionsthm} which state that any filled Cassini set can be thought of as both a  union of intersections of a family of filled guide Cassini sets, or an intersection of unions of those same sets.  See
        Figure~\ref{guidecassinifig} for examples of such sets, and see Figure~\ref{unionsandintersectionsfig} for an illustration of the two construction processes.

        \begin{figure}
            \begin{picture}(130,130)
                \put(0,0){
                    \includegraphics[scale = .5, clip = true, draft = false]{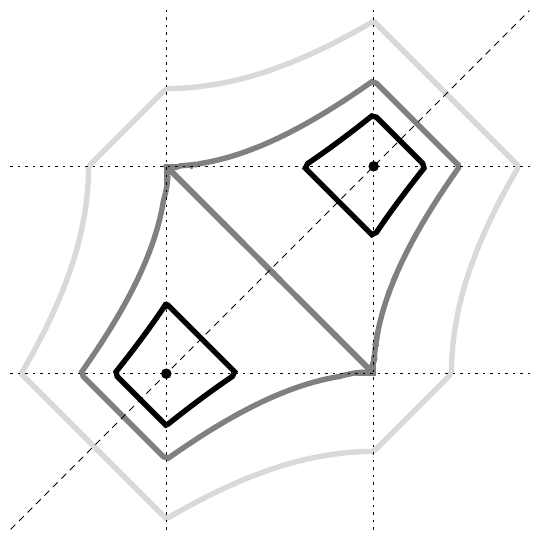}
                }
                \put(38,44){\small{$q$}}
                \put(95,85){\small{$p$}}
            \end{picture}
        
            \caption{
                Guide Cassini sets with $r < r^*$ (black), $r = r^*$ (dark gray), and $r > r^*$ (light gray).
            } \label{guidecassinifig}
        \end{figure}

          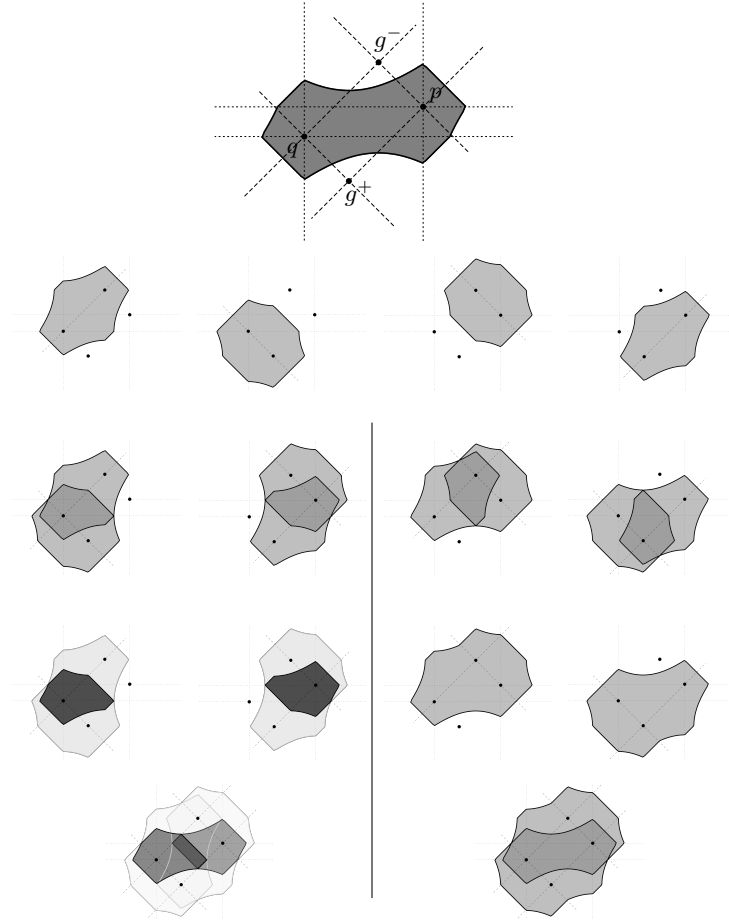
\begin{figure}
            \begin{picture}(280,360)
                \put(0,200){
                    \includegraphics[scale = .25, clip = true, draft = false]{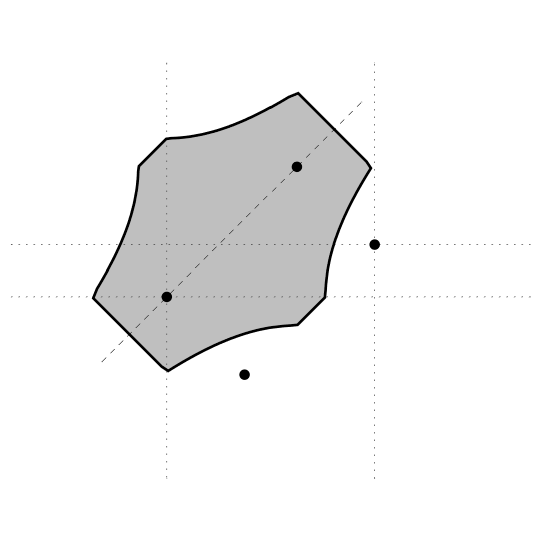}
                }    
                \put(70,200){
                    \includegraphics[scale = .25, clip = true, draft = false]{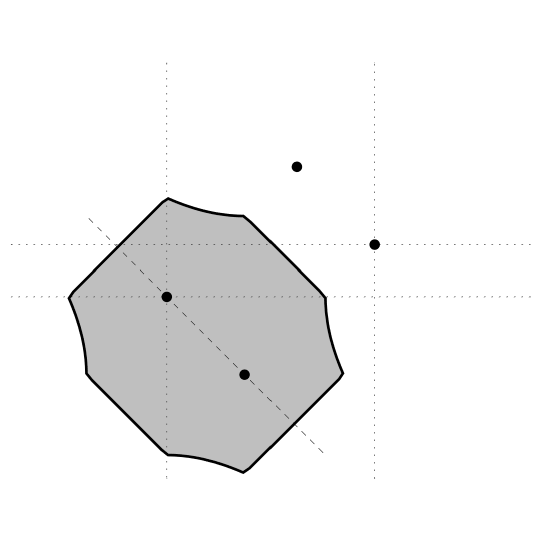}
                }
                \put(140,200){
                    \includegraphics[scale = .25, clip = true, draft = false]{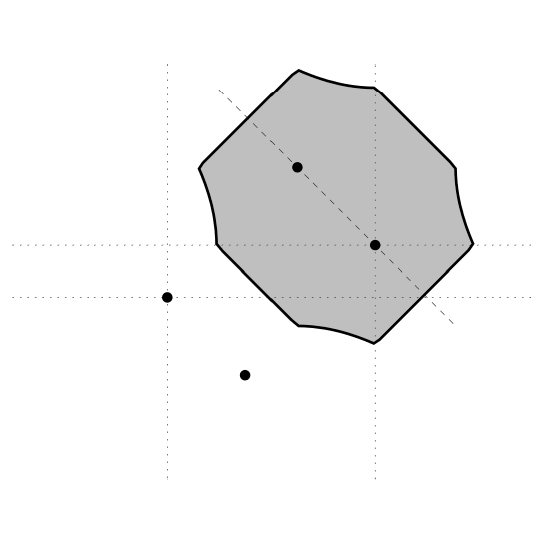}
                }    
                \put(210,200){
                    \includegraphics[scale = .25, clip = true, draft = false]{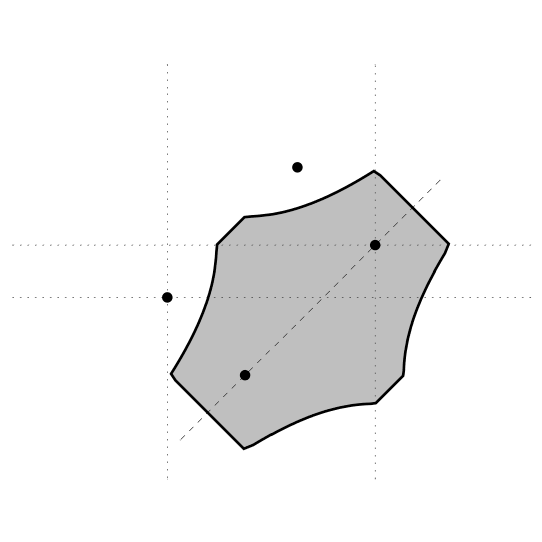}
                }
                \put(0,130){
                    \includegraphics[scale = .25, clip = true, draft = false]{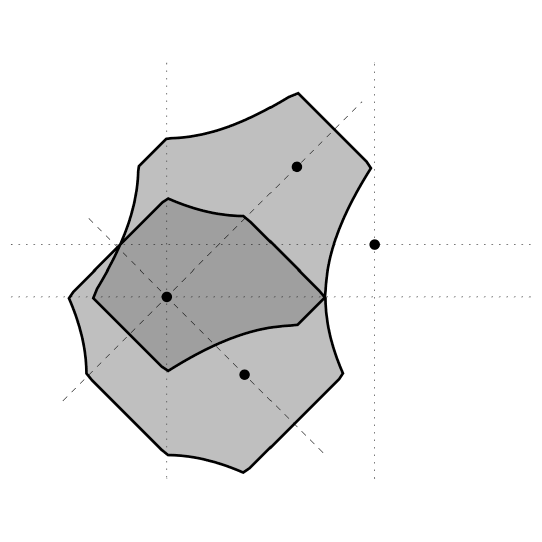}
                }    
                \put(70,130){
                    \includegraphics[scale = .25, clip = true, draft = false]{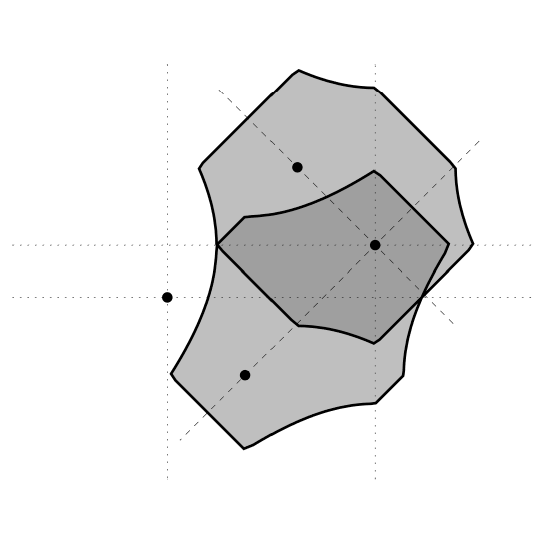}
                }
                \put(140,130){
                    \includegraphics[scale = .25, clip = true, draft = false]{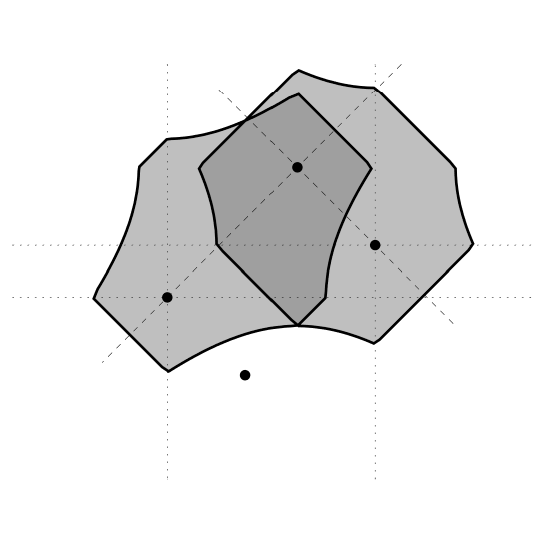}
                }    
                \put(210,130){
                    \includegraphics[scale = .25, clip = true, draft = false]{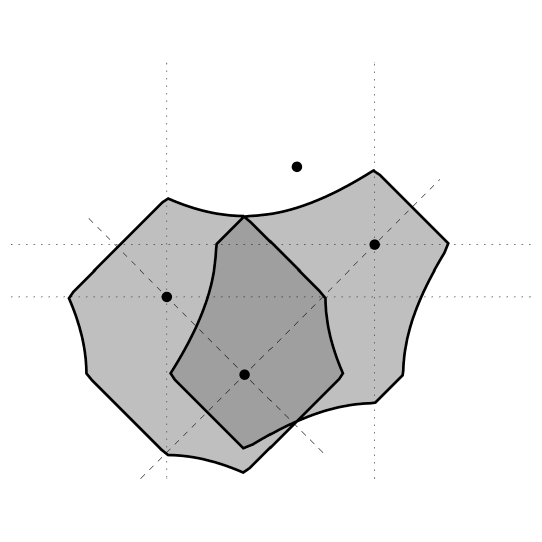}
                }
                \put(140,60){
                    \includegraphics[scale = .25, clip = true, draft = false]{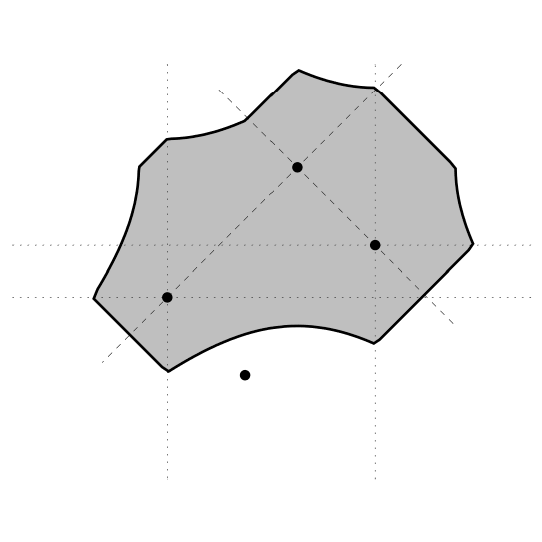}
                }    
                \put(210,60){
                    \includegraphics[scale = .25, clip = true, draft = false]{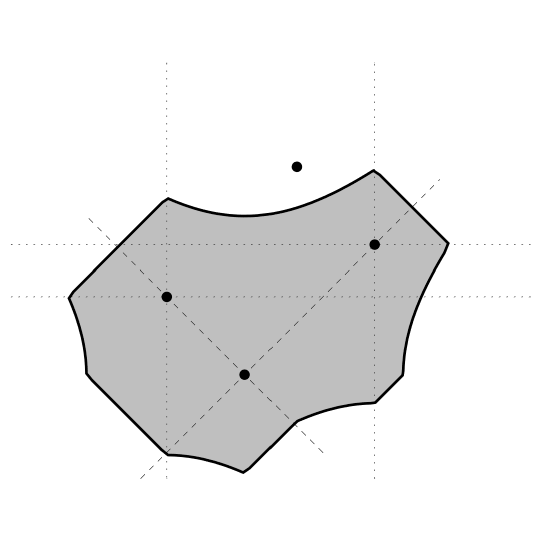}
                }
                \put(0,60){
                    \includegraphics[scale = .25, clip = true, draft = false]{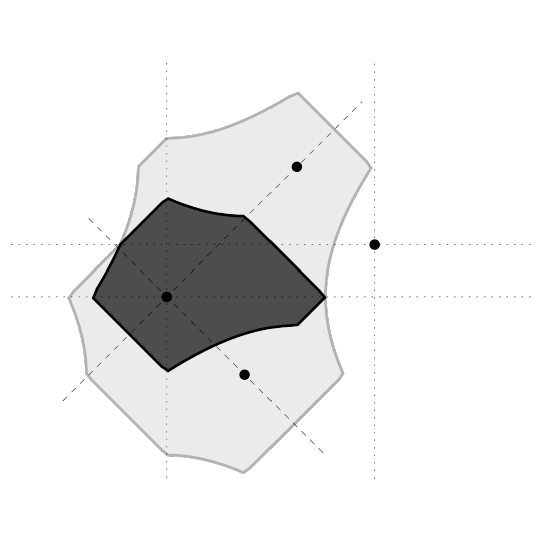}
                }    
                \put(70,60){
                    \includegraphics[scale = .25, clip = true, draft = false]{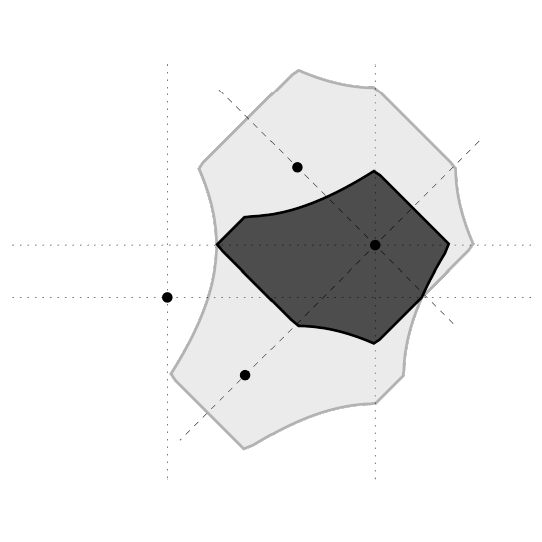}
                }
                \put(35,0){
                    \includegraphics[scale = .25, clip = true, draft = false]{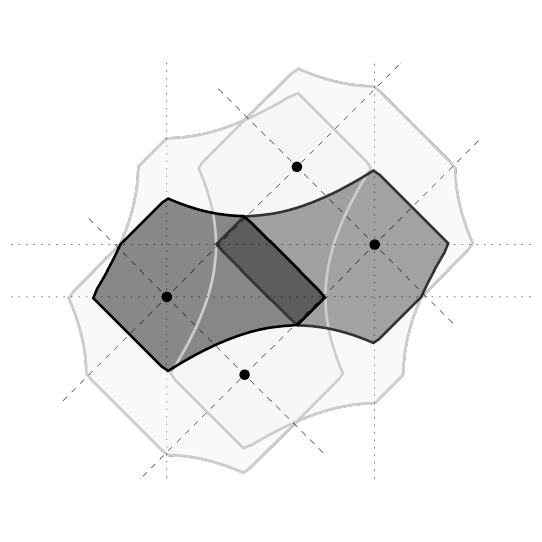}
                }
                \put(175,0){
                    \includegraphics[scale = .25, clip = true, draft = false]{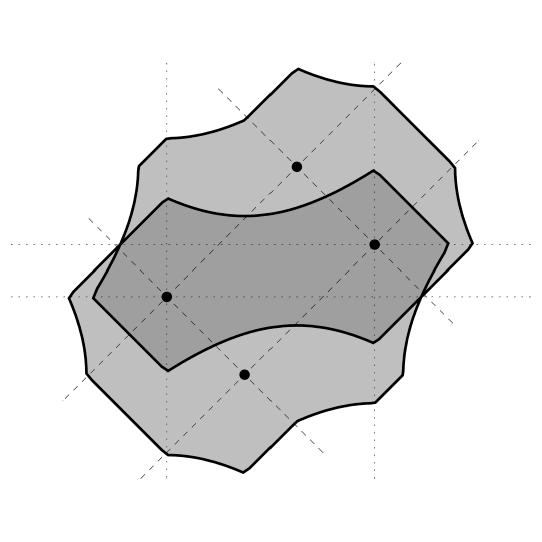}
                }
                \put(75,250){
                    \includegraphics[scale = .45, clip = true, draft = false]{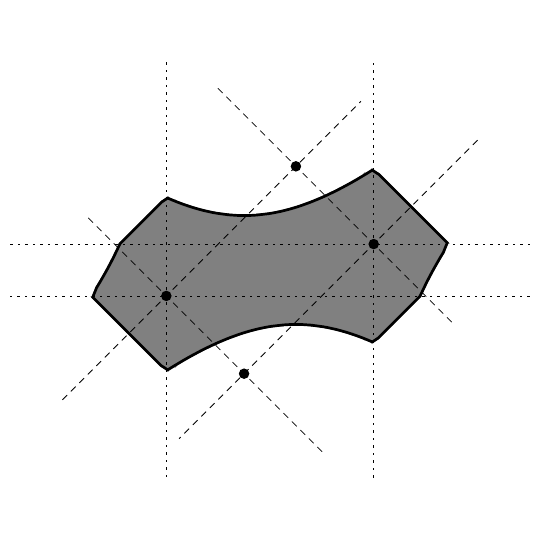}
                    \put(-34,68){\small{$p$}}
                    \put(-88,47){\small{$q$}}
                    \put(-55,87){\small{$g^-$}}
                    \put(-66,29){\small{$g^+$}}
                }
                \put(140,15){\line(0,1){180}}
            \end{picture}
            
            \caption{
                Characterizing a filled Cassini set in terms of speciffic filled Cassini sets.  To produce the filled Cassini set at the top, begin with the four sets in the second row.  Then a union of intersections is shown on the left (Theorem~\ref{unionofintersectionsthm}) and an intersection of unions is shown on the right (Theorem~\ref{intersectionofunionsthm}).  For these sets, $p = (4, 1)$, $q = -p$, and $r = 6$.
            } \label{unionsandintersectionsfig}
        \end{figure}

        \begin{main} \label{unionofintersectionsthm}
            Let $p, q \in (\mathbb{R}^2, d)$, let $r \geq 0$, and let $g^+$ and $g^-$ be the guide complements of $p$ and $q$.  Then
            \begin{equation} \label{unionofintersectionseq}
                L(p,q;r) = \Bigl[L(p,g^{+};r) \cap L(p,g^{-};r)\Bigr] \cup \Bigl[L(q,g^{+};r) \cap L(q,g^{-};r)\Bigr].
            \end{equation}
        \end{main}

        \begin{proof}
            This proof proceeds in two steps.  Suppose first that $x$ is an element of the set on the right of Equation~\eqref{unionofintersectionseq}.
            Specifically, suppose $x \in L(p,g^{+};r) \cap L(p,g^{-};r)$.  Then
            \begin{equation*}
                d(x, p) \cdot d(x, g^+) < r^2\
                \mbox{and}\ d(x, p) \cdot d(x, g^-) < r^2.
            \end{equation*}
            By Lemma~\ref{guideHlemma}, using $q$ for $a$ and $g^+$ and $g^-$ for $b$ and $c$, $d(x, q) \leq d(x, g^+)$ or $d(x, q) \leq d(x, g^-)$.  Pairing whichever inequality is satisfied with the corresponding inequality above, it follows that
            \begin{equation*}
                d(x, p) \cdot d(x, q) < r^2
            \end{equation*}
            and so $x \in L(p, q; r)$.
            
            An analogous argument works for $x \in L(q,g^{+};r) \cap L(q,g^{-};r)$, and therefore
            \begin{equation} \label{superseteqn}
                L(p,q;r) \supseteq \Bigl[L(p,g^{+};r) \cap L(p,g^{-};r)\Bigr] \cup \Bigl[L(q,g^{+};r) \cap L(q,g^{-};r)\Bigr].
            \end{equation}
            
            Second, suppose $x \in L(p, q; r)$, so $d(x, p) \cdot d(x, q) < r^2$.  Applying Lemma~\ref{guideHlemma} using $g^+$ for $a$, and $p$ and $q$ for $b$ and $c$, and substituting the resulting inequalities, at least one of the following is true:
            \begin{equation} \label{alphaeqn}
                d(x, g^+) \cdot d(x, q) < r^2
            \end{equation}
            or
            \begin{equation} \label{betaeqn}
                d(x, p) \cdot d(x, g^+) < r^2.
            \end{equation}
            Also, applying Lemma~\ref{guideHlemma}, this time using $g^-$ for $a$, at least one of the following is true:
            \begin{equation} \label{gammaeqn}
                d(x, g^-) \cdot d(x, q) < r^2
            \end{equation}
            or
            \begin{equation} \label{deltaeqn}
                d(x, p) \cdot d(x, g^-) < r^2.
            \end{equation}
            Note additionally that if $d(x, p) \leq d(x, q)$ then Inequality~\eqref{alphaeqn} implies Inequality~\eqref{betaeqn} and Inequality~\eqref{gammaeqn} implies Inequality~\eqref{deltaeqn}.  Hence, if $d(x, p) \leq d(x, q)$ then Inequalities~\eqref{betaeqn} and \eqref{deltaeqn} must be true which means $x \in L(p, g^+; r) \cap L(p, g^-; r)$.  By a similar argument, if $d(x, q) \leq d(x, p)$ then Inequalities~\eqref{alphaeqn} and \eqref{gammaeqn} must be true which means $x \in L(q, g^+; r) \cap L(q, g^-; r)$.  Combining these, it follows that
            \begin{equation} \label{subseteqn}
                L(p, q; r) \subseteq \Bigl[L(p,g^{+};r) \cap L(p,g^{-};r)\Bigr] \cup \Bigl[L(q,g^{+};r) \cap L(q,g^{-};r)\Bigr].
            \end{equation}

            Combining \eqref{superseteqn} and \eqref{subseteqn} establishes the result.
   
        \end{proof}

        Theorem~\ref{unionofintersectionsthm} shows that $L(p, q; r)$ can be expressed as a union of intersections.  The set $L(p, q; r)$ can also be expressed as an intersection of unions.

        \begin{main} \label{intersectionofunionsthm}
            Let $p, q \in (\mathbb{R}^2, d)$, let $r \geq 0$, and let $g^+$ and $g^-$ be the guide complements of $p$ and $q$.  Then
            \begin{equation} \label{intersectionofunionseq}
                L(p,q;r) = \Bigl[L(p,g^{+};r) \cup L(q,g^{+};r)\Bigr]
                    \cap \Bigl[L(p,g^{-};r) \cup L(q,g^{-};r)\Bigr].
            \end{equation}
        \end{main}

        We prove this in a fashion analogous to the proof of Theorem~\ref{unionofintersectionsthm}.
        \begin{proof}
            First, suppose $x$ is an element of the set on the right of Equation~\eqref{intersectionofunionseq}.  Then $d(x, p) \cdot d(x, g^+) < r^2$ or $d(x, q) \cdot d(x, g^+) < r^2$, and also $d(x, p) \cdot d(x, g^-) < r^2$ or $d(x, q) \cdot d(x, g^-) < r^2$.
            
            Suppose $d(x, p) \leq d(x, q)$.  By Lemma~\ref{guideHlemma}, using $q$ for $a$ and $g^+$ and $g^-$ for $b$ and $c$, it follows that $d(x, q) \leq d(x, g^+)$ or $d(x, q) \leq d(x, g^-)$.  Substituting whichever is true, it follows that $d(x, p) \cdot d(x, q) < r^2$ and so $x \in L(p, q; r)$.  An analogous argument works if $d(x, q) \leq d(x, p)$ and so
            \begin{equation} \label{superset2eqn}
                L(p,q;r) \supseteq \Bigl[L(p,g^{+};r) \cup L(q,g^{+};r)\Bigr]
                    \cap \Bigl[L(p,g^{-};r) \cup L(q,g^{-};r)\Bigr].
            \end{equation}
            
            In the other direction, suppose $x \in L(p, q; r)$, so $d(x,p) \cdot d(x,q) < r^2$.  By Lemma~\ref{guideHlemma}, with $g^+$ for $a$ and $p$ and $q$ for $b$ and $c$, it follows that $d(x, g^+) \leq d(x, p)$ or $d(x, g^+) \leq d(x, q)$.  If the first is true, then
            \begin{equation*}
                d(x, q) \cdot d(x, g^+) \leq d(x, q) \cdot d(x, p) < r^2
            \end{equation*}
            so $x \in L(q, g^+; r)$.  Similarly, if the second is true, then $x \in L(p, g^+; r)$.  Together, these imply that $x \in L(p,g^{+};r) \cup L(q,g^{+};r)$.
            
            An analogous argument can be made to show $x \in L(p,g^{-};r) \cup L(q,g^{-};r)$ and since $x$ lies in both sets, it lies in their intersection, implying that
            \begin{equation} \label{subset2eqn}
                L(p,q;r) \subseteq \Bigl[L(p,g^{+};r) \cup L(q,g^{+};r)\Bigr]
                    \cap \Bigl[L(p,g^{-};r) \cup L(q,g^{-};r)\Bigr].
            \end{equation}
            Combining \eqref{superset2eqn} and \eqref{subset2eqn} establishes the result.
        \end{proof}

        Given the similarities in the proofs of Theorems~\ref{unionofintersectionsthm} and \ref{intersectionofunionsthm}, it may seem that these characterizations could be shown to be equivalent through general set manipulation.  In fact, using the distributive laws for unions and intersections, the following new relationships are established.

        \begin{corollary}
            Let $p, q \in (\mathbb{R}^2, d)$, let $r \geq 0$, and let $g^+$ and $g^-$ be the guide complements of $p$ and $q$.  Then
            \[
                L(p, q; r)
                    \subseteq [L(p, g^+; r) \cup L(q, g^-; r)]
                        \cap
                        [L(p, g^-; r) \cup L(q, g^+; r)]
            \]
            and
            \[
                [L(p, g^+; r) \cap L(q, g^-; r)]
                    \cup
                    [L(p, g^-; r) \cap L(q, g^+; r)]
                    \subseteq
                        L(p, q; r).
            \]
        \end{corollary}
        In general, these relationships are proper.  To clarify the situation, it is worth noting that $p$ and $q$ are the guide complements to $g^+$ and $g^-$, and using Theorems~\ref{unionofintersectionsthm} and \ref{intersectionofunionsthm} from this alternate perspective, it follows that in fact
        \begin{corollary}
            Let $p, q \in (\mathbb{R}^2, d)$, let $r \geq 0$, and let $g^+$ and $g^-$ be the guide complements of $p$ and $q$.  Then
            \[
                [L(p, g^+; r) \cup L(q, g^-; r)]
                    \cap
                    [L(p, g^-; r) \cup L(q, g^+; r)]
                    = L(p, q; r) \cup L(g^+, g^-; r)    
            \]
            and
            \[
                [L(p, g^+; r) \cap L(q, g^-; r)]
                    \cup
                    [L(p, g^-; r) \cap L(q, g^+; r)]
                    =
                        L(p, q; r) \cap L(g^+, g^-; r).
            \]
        \end{corollary}
        We leave the proofs of these results to the interested reader.

 \section{Concluding remarks} \label{conclusionsec}

There are a number of aspects of this work that we feel warrant further consideration, and we outline our thoughts in three interconnected directions here.

First, as mentioned in the introduction, the work in \cite{BCFHMNSTV} characterizing Apollonian sets as the union of guide Apollonian sets was surprising.  Finding a similar characterization for Cassini sets indicates that there may be a much more general result that would unify these examples.  A simple place to start would be to confirm that such a characterization exists for the taxicab conic sections, and some brief sketches indicate that such characterizations do indeed exist.  Additionally, the first author was able to explore a wide variety of sets, each defined as the locus of points satisfying an equation involving the distances to a pair of focal points, where similar characterizations seem to hold.

We expect that the proof of a general characterization would be similar to the work in \cite{BCFHMNSTV} and our work here.  The unifying result seems to be Lemma~\ref{guideHlemma} which provides key distance inequalities for all points in the plane.  As such, we feel that this lemma is of fundamental importance to taxicab geometry.

Another intriguing aspect that appears here, but not in \cite{BCFHMNSTV}, is the dual nature of Theorems~\ref{unionofintersectionsthm} and \ref{intersectionofunionsthm}.  Weather this duality is a consequence of the symmetry in the defining equation or something else is not clear at this point, but generalizing the sets which enjoy such a characterization may shed light on the situation.

Finally, in the recent work \cite{FHS}, taxicab conic sections were explored through the lens of slicing cones in taxicab 3-space, resulting in geometric characterizations of these sets.  While differing somewhat from the more traditional approach of using the distance formulations for the various conic sections, this work illustrates that often, structures in taxicab space can be characterized geometrically.  This more geometric perspective was also advanced for Apollonian sets in \cite{BCFHMNSTV}, and continued here for Cassini sets.  All of these projects have demonstrated that taxicab space exhibits a beautiful geometry that is often overlooked when attention is restricted to its purely analytic characteristics.


    \bibliographystyle{amsalpha}
    \bibliography{taxicab-cassini.bib}
\end{document}